\documentclass[12pt]{amsart}

\usepackage{latexsym, amssymb}

\newcommand{\be}{\begin{equation}}
\newcommand{\ee}{\end{equation}}
\newcommand{\beq}{\begin{eqnarray}}
\newcommand{\eeq}{\end{eqnarray}}

\newtheorem{prop}{Proposition}[section]
\newtheorem{theo}[prop]{Theorem}
\newtheorem{lemm}[prop]{Lemma}

\newtheorem{rema}[prop]{Remark}

\newtheorem{defi}[prop]{Definition}

\def\begeq{\begin{equation}}
\def\endeq{\end{equation}}
\def\p{\partial}

\begin{document}

\title {static flow on complete noncompact manifolds I: short-time existence and asymptotic expansions at conformal infinity}

\begin{abstract}
In this paper, we study short-time existence of static flow on complete noncompact asymptotically static manifolds from the point of view that the stationary points of the evolution equations can be interpreted as static solutions of the Einstein vacuum equations with negative cosmological constant. For a static vacuum $(M^n,g,V),$ we also compute the asymptotic expansions of $g$ and $V$ at conformal infinity.
\end{abstract}

\keywords{static flow, asymptotical static, asymptotical hyperbolic,
complete noncompact manifolds, short-time existence, asymptotic expansions }
\renewcommand{\subjclassname}{\textup{2000} Mathematics Subject Classification}
 \subjclass[2000]{Primary 53C25; Secondary 58J05}

\author{Xue Hu $^\dag$,  Yuguang Shi$^\dag$}

\address{Xue Hu, Key Laboratory of Pure and Applied mathematics, School of Mathematics Science, Peking University,
Beijing, 100871, P.R. China.} \email{huxue@math.pku.edu.cn}

\address{Yuguang Shi, Key Laboratory of Pure and Applied mathematics, School of Mathematics Science, Peking University; Beijing International Center for Mathematical Research, Beijing, 100871, P.R. China.} \email{ygshi@math.pku.edu.cn}

\thanks{$^\dag$ Research partially supported by NSF grant of China 10725101 and 10990013.}

\date{2011}
\maketitle

\markboth{Xue Hu,  Yuguang Shi}{}

\section{Introduction}
Geometric flow equations play an important role in geometric analysis nowadays. For example, the Ricci flow has contributed to the thorough resolution of the Poincar$\acute{e}$ and Thurston conjectures \cite{P1}, \cite{P2}, and the diffeomorphic $\frac{1}{4}$-pinched sphere theorem \cite{BS}, \cite{BS1}, while the inverse mean curvature flow has yielded a proof of the Riemannian Penrose conjecture \cite{HI}, \cite{HI1}. From the latter, naturally, the question of whether these powerful geometric flow equations can be used in physics has become common concern both in differential geometry and physics.

The class of static spacetimes is the most simple and interesting object in general relativity. List has proposed a geometric flow from the point of view that the stationary points of the evolution equations can be interpreted as static solutions of the Einstein vacuum equations \cite{BL}.

In this note, we mainly generalize List's results to the case with negative cosmological constant. In this case, the spacetime $(X^{n+1},\tilde{h})$ satisfies the equation
\begin{equation}\label{negeinstein}
G+\Lambda \tilde{h}=8\pi T
\end{equation}
where $G=Ric(\tilde{h})-\frac{1}{2}R(\tilde{h}),$ is the Einstein tensor of $(X^{n+1},\tilde{h}),$ $T$ is the energy-momentum tensor, and $\Lambda$ is negative cosmological constant. To find solutions to (\ref{negeinstein}) is a central issue in physics but, in general, the equation is hard to solve, thus special cases are considered. For $T=0,$ we call a solution to (\ref{negeinstein}) a vacuum spacetime. Without loss of generality, we normalize the constant $\Lambda$ for a vacuum spacetime such that \begin{equation}\label{einstein}
Ric(\tilde{h})=-n\tilde{h}.
\end{equation}
Much like List did, we consider spacetimes which have some symmetry. In the following, we assume $n\geq3.$
\begin{defi}\label{static}
A Lorentzian manifold $(X^{n+1}, \tilde{h})$ is said to be stationary, if there exists a 1-parameter group of isometries with timelike orbits. If in addition, there exists a hypersurface $M^n$ which is orthogonal to these orbits and therefore spacelike, $(X^{n+1}, \tilde{h})$ is said to be static.
\end{defi}
\begin{rema}
For a 1-parameter group of isometries with timelike orbits, equivalent is the existence of a timelike Killing vector field $\xi.$
\end{rema}

\begin{rema}
Static spacetime metric splits as a warped product of $\mathbb{R}$ and a Riemannian manifold, i.e., $$X^{n+1}=\mathbb{R}\times M^n, \tilde{h}=-V^{2}dt^{2}+g,$$ where $(M^{n},g)$ is a Riemannian manifold and $V:=\sqrt{-\tilde{h}(\xi,\xi)}$ is a positive function on $M^n.$
\end{rema}

If we apply the Einstein vacuum equation (\ref{einstein}) to the static spacetime metric $\tilde{h}$, we obtain the equations
\begin{equation}\label{1}
Ric(g)+ng=V^{-1}\nabla_{g}^{2}V,
\end{equation}
and
\begin{equation}\label{2}
\Delta_{g}V=nV.
\end{equation}
\begin{defi}
A Riemannian manifold $(M^n,g)$ with a positive function $V$ on $M^n$ satisfying (\ref{1}) and (\ref{2}) is called static Einstein vacuum, denoted by $(M^n,g,V).$
\end{defi}

Examples of static Einstein vacuum are the Anti-de Sitter and Schwarzschild-AdS metrics. For more details, please refer to \cite{CS}, \cite{Q} and \cite{Wa} and the references therein.

We introduce the following static flow for $g$ and $V$ on $M^n:$
\begin{equation}\label{originalequation}
\left\{
\begin{array}{ll}
    \frac{\partial}{\partial t} g= -2Ric(g)-2ng+2V^{-1}\nabla_{g}^{2}V  \\
    \frac{\partial}{\partial t} V= \Delta_{g} V-nV.\\
\end{array}
\right.
\end{equation}

In \cite{BL}, List considered static Einstein vacuum with vanishing cosmological constant, that is, a triple $(M^n,g,V)$ satisfies
\begin{equation}
Ric(g)=V^{-1}\nabla_{g}^{2}V\nonumber
\end{equation}
and
\begin{equation}
\Delta_{g}V=0.\nonumber
\end{equation}
He took $u=\ln V$ and took conformal transformation $\tilde{g}=e^{\frac{2}{n-2}u}\cdot g,$ in this way, the static Einstein vacuum equation $\Delta_{g}V=0$ became $\Delta_{\tilde{g}}u=0,$ while $Ric(g)=V^{-1}\nabla_{g}^{2}V$ became $Ric(\tilde{g})=\frac{n-1}{n-2}\nabla_{\tilde{g}} u\otimes\nabla_{\tilde{g}} u,$ involving second derivatives of $\tilde{g}$ but merely first derivatives of $u.$  Thus List proposed the extended Ricci flow system
\begin{equation}\label{list}
\left\{
\begin{array}{ll}
    \frac{\partial}{\partial t} g= -2Ric(g)+4du\otimes du,  \\
    \frac{\partial}{\partial t} u= \Delta_{g} u.\\
\end{array}
\right.
\end{equation}
Using DeTurck's trick, this system could be changed to be strictly quasi-linear parabolic system and then List applied standard parabolic theory as well as Shi's ideas \cite{Shi} to get the short-time existence result.

We point out here that List's system of flow equations is in fact the pullback by a certain diffeomorphism of a class of Ricci flows in one higher dimension(Please see \cite{AW}). Comparing with List's flow, our static flow (\ref{originalequation}) comes directly from the one-higher dimensional Ricci flow, more specifically, the Ricci flow of a warped product. Notice that for static Einstein vacuum $(M^n,g,V),$ the Riemannian metric $h=V^{2}d\theta^{2}+g$ can be viewed as an Einstein metric on $\mathbb{S}^{1}\times M^n.$ Due to this fact, Anderson, Chru$\acute{s}$ciel and Delay could use the result for conformally compact Einstein manifolds to study non-trivial, static, geodesically complete, vacuum space-times with a negative cosmological constant. Please see \cite{ACD} and \cite{ACD1}. Meanwhile, the Ricci tensor of $h$ is given by
$$
Ric(h)=
\left(
  \begin{array}{cc}
    -V^{-1}\Delta_{g}V & 0 \\
    0 & Ric(g)-V^{-1}\nabla_{g}^{2}V \\
  \end{array}
\right).
$$
Thus we can see that the normalized Ricci flow equation
\begin{equation}\label{nrf}
\frac{\partial}{\partial t}h=-2Ric(h)-2nh\nonumber
\end{equation}
induces the static flow (\ref{originalequation}).

However, taking into account factor such as the use of the flow, for negative cosmological constant case, we don't employ the means List used. List's flow (\ref{list}) preserves two classes of asymptotically flat solutions assuming a uniform curvature bound (see \cite{BL} Theorem 9.5 and Theorem 9.7). Similarly, we hope to propose a kind of flow which preserves some classes of asymptotically hyperbolic(see definition\ref{AH}) solutions under some circumstances. In the asymptotically flat category, though List's flow is an extended Ricci flow, we can see that the conformal transformation doesn't bring any trouble in the preservation of asymptotically flat structure. But even for Anti-de Sitter spacetime, $$(\mathbb{R}\times \mathbb{H}^{n}, -\cosh^{2}rdt^{2}+dr^{2}+\sinh^{2}rdw^{2}),$$ that is, $$(M^n,g,V)=(\mathbb{H}^{n},dr^{2}+\sinh^{2}rdw^{2},\cosh r),$$ if we take the same conformal transformation, we find that $\tilde{g}=e^{\frac{2}{n-2}u}\cdot g$ where $u=\ln V$ induces a new metric $\tilde{g}=(\cosh r)^{\frac{2}{n-2}}g_{\mathbb{H}^n},$ which is obviously not asymptotically hyperbolic. In other words, List's conformal transformation doesn't preserve asymptotically hyperbolic structure.

Now we give some basic notions in conformally compact geometry that we may use later.

Suppose that $X^{n+1}$ is a smooth manifold with boundary $\partial
X^{n+1}= Y^n$. A defining function $\tau$ of the boundary $Y^n$ in
$X^{n+1}$ is a smooth function on $X^{n+1}$ such that
\begin{enumerate}
  \item $\tau > 0$ in $X^{n+1}$;
  \item $\tau = 0$ on $Y^n$;
  \item $d\tau \neq 0$ on $Y^n$.
\end{enumerate}
A complete Riemannian metric $h$ on $X^{n+1}$ is
conformally compact if $(\bar{X}^{n+1}, \bar{h}=\tau^2h)$ is a compact Riemannian manifold for a smooth
defining function $\tau$ of the boundary $Y^n$ in $X^{n+1}.$ If $\bar{h}$ is $C^{k, \mu},$ $h$ is said to be
conformally compact of regularity $C^{k, \mu}.$
The restriction of $\bar{h}$ to $TY^n$ rescales upon changing $\tau,$ so defines invariantly a conformal class of metrics on $Y^n.$ $(Y^{n}, [\bar{h}\mid_{TY^n}])$ is called the conformal infinity of the conformally compact manifold $(X^{n+1}, h).$
\begin{defi}\label{AH}
A complete noncompact Riemannian manifold $(X^{n+1}, h)$ is called asymptotically hyperbolic($AH$) of order $a$ if
\begin{equation}
||Rm-\mathbf{K}||_{h}\leq C e^{-a \rho},\nonumber
\end{equation}
where $Rm$ denotes the Riemann curvature tensor of the metric $h$
and $\mathbf{K}$ the constant curvature tensor of $-1$, i.e.,
$\mathbf{K}_{ijkl}=-(h_{ik}h_{jl}-h_{il}h_{jk})$; $\rho$ is the
distance function to a fixed point in $X^{n+1}$ with respect to $h$;
and $C$ is a positive constant independent of $\rho$.
\end{defi}
\begin{rema}
$AH$ order $a=2$ is important and interesting in mathematics and physics. On one hand, roughly speaking, there is rigidity when $a>2$ (see \cite{HQS} and the references therein); On the other hand, $C^2$ conformally compact Einstein manifolds are
$AH$ of order $2.$ In order to avoid the complexity of the end structure of a hyperbolic manifold, we need the concept of an essential set here. Please see \cite{HQS} Definition 1.1 for the definition of an essential set. In Gicquaud's PhD thesis, he proved that if a complete noncompact manifold is $C^2$ conformally compact then it contains essential sets(\cite{GR} Lemma 2.5.11 and
Corollary 2.5.12).
\end{rema}
The recent work of \cite{B} and \cite{QSW} investigated the behavior of normalized Ricci flow on $AH$ manifolds, while it took care of conformal infinities. By virtue of the ideas in these two papers, we give the following
\begin{defi}\label{AS}
Suppose $(M^n,g)$ is a smooth complete noncompact Riemannian manifold with an essential set and $V$ is a positive function on $M^n.$ Then a triple $(M^{n},g,V)$ is called asymptotically static(AS) of order $a\geq 2$ if $(\mathbb{S}^{1}\times M^n, h=V^{2}d\theta^{2}+g)$ is $AH$ of order $2,$ and
\begin{equation}\label{hah}
||\nabla_{h}Ric(h)||_{h}(\theta, x)\leq C e^{-a\rho(\theta,x)}.\nonumber
\end{equation}
where $\rho$ is the distance function to a fixed point in $\mathbb{S}^{1}\times M^n$ with respect to $h$;
and $C$ is a positive constant independent of $\rho.$
\end{defi}

If we compute the curvature tensor of the warped product in terms of Gauss-Codazzi equations, we have the following
\begin{prop}\label{equivalent}
If $(M^{n},g,V)$ is AS of order $a\geq2$, then
\begin{enumerate}
\item $(M^n,g)$ is AH of order $2;$
\item $||Ric(g)+ng-V^{-1}\nabla_{g}^{2}V||_{g}\leq Ce^{-2r};$
\item $||\nabla_{g}(V^{-1}\Delta_{g}V)||_{g}\leq Ce^{-ar}$ and $||\nabla_{g}(Ric(g)-V^{-1}\nabla_{g}^{2}V)||_{g}\leq Ce^{-ar};$
\end{enumerate}
where $r$ is the distance function to a fixed point in $M^n$ with respect to $g.$
\end{prop}

Now, we state the main results:
\begin{theo}\label{mainresult}
Suppose $(M^n,g_{0},V_{0})$ is asymptotically static of order $a\geq2$ and $\|e^{-r}V_{0}\|_{C^{2+\alpha}}\leq C$ where $\alpha\in(0,1).$  Then for any $\epsilon>0,$ there exists $T_{0}=T_{0}(n,C,\epsilon)$ such that the static flow (\ref{originalequation})
with initial data $g(0)=g_{0}$ and $V(0)=V_{0}$ has a smooth solution $(g,V)(x,t)$ on $M\times[0,T_{0}]$ satisfying $$\parallel g-g_{0}\parallel_{g_{0}}(x)+\parallel \nabla_{g_{0}}g\parallel_{g_{0}}(x)\leq \epsilon e^{-2r(x)}$$
for all $(x,t)\in M\times[0,T_{0}].$
\end{theo}

\begin{rema}
Since the static flow for $(M^n,g,V)$ comes from the normalized Ricci flow for $(\mathbb{S}^{1}\times M^n, h=V^{2}d\theta^{2}+g),$ we take advantage of the results in \cite{CZ} to obtain the uniqueness of the static flow for curvature bounded solutions. In Theorem \ref{mainresult}, if the initial metric is $C^{2+\alpha},$ the solution $(g,V)$ to the static flow we get in Theorem \ref{mainresult} is still $C^{2+\alpha},$ then the curvature for $h$ is bounded. Hence we get the uniqueness of the static flow.
\end{rema}
Given a conformally compact, asymptotically hyerbolic manifold $(X^{n+1},h)$ and a representative $\hat{h}$ in $[\hat{h}]$ on the conformal infinity $Y^n,$ there is a unique determined defining function $\tau$ such that, in a neighborhood of the boundary $[0, \delta)\times Y^{n}\subset \bar{X}^{n+1},$ $h$ has the form
\begin{equation}
h=\tau^{-2}(d\tau^2+h_{\tau}),\nonumber
\end{equation}
where $h_{\tau}$ is a $1-$parameter family of metrics on $Y^n.$ We call this $\tau$ the special defining function associated with $\hat{h}.$
In \cite{G}, Graham showed the asymptotic expansion of a conformally compact Einstein manifold with respect to the special defining function, which is associated with a conformal infinity. Naturally, we come to the question that whether we can do the same thing for static Einstein vacuum. But this will arose another question that how $V$ behaves near the infinity. In \cite{Wa}, X.D. Wang has proved a uniqueness theorem of AdS spacetime, there he required $(M^n,g)$ to be conformally compact and $V^{-1}$ to be a defining function for $(M^n,g).$ Inspired by this idea, we require that $V$ be with the growth of the inverse of a special defining function near conformal infinity, then we have
\begin{theo}\label{expansion}
Suppose that $(M^n, g, V)$ is a static Einstein vacuum, that $(M^n, g)$ is an asymptotically hyperbolic manifold with the conformal infinity $(N^{n-1}, [\hat{g}])$ and that $\tau$ is the special defining function associated with a metric $\hat{g}\in [\hat{g}].$ Assume also that $V$ is with the growth of $\frac{1}{\tau}$ near conformal infinity. If $g=\tau^{-2}(d\tau^2+g_{\tau}),$ then we have
\begin{equation}
g_{\tau}=\hat{g}+g^{(2)}\tau^{2}+(even~powers~of~\tau)+g^{(2l)}\tau^{2l}+\cdots\nonumber
\end{equation}
\begin{equation}
V=\frac{1}{\tau}+V^{(1)}\tau+(odd~powers~of~\tau)+V^{(2l-1)}\tau^{2l-1}+\cdots\nonumber
\end{equation}
where $g^{(2l)}$ and $V^{(2l-1)}$ are uniquely determined by $\hat{g}$ for $2l\leq n-1.$

\end{theo}

This paper would be the first one among a series of papers, which we are about to work on. In our sequel paper, we will pay more attention to the conformal category. We will use Theorem\ref{expansion} and the methods in \cite{QSW} to prove existence result of the static metric, and discuss the behavior of the conformal infinity in this flow and other related problems. This paper is organized as follows: In section 2 we prove the short-time existence of the static flow. In section 3, we compute the asymptotic expansions at conformal infinity.\\

{\bf Acknowledgements}
The authors are grateful to Professor Xiaodong Wang, Professor Jie Qing and Professor Romain Gicquaud for their interests in this work and many enlightening discussions. The authors are also indebted to Professor Romain Gicquaud to show us the references \cite{ACD} and \cite{ACD1}, as well as correct some mistakes. The authors would like to express our gratitude for Professor Eric Woolgar to point out a fact that we didn't know and show the citation \cite{AW}.

\section {short-time existence}
In this section, we mainly prove the short-time existence of our static flow. We follow the ideas of Shi\cite{Shi} and List\cite{BL} to modify the static flow and derive some formulas from the altered flow equations, and then use a maximum principle by Ecker and Huisken to get our basic estimates. Next, we construct a Banach space and an operator, together with standard parabolic theory for linear equation, we can check the fulfillment of the Schauder fixed point theorem, which promises the short-time existence of static flow.

We begin with the proof of Proposition \ref{equivalent} and then we can use it in the following estimates.
\begin{proof}[Proof of Proposition \ref{equivalent}]
For any $x\in M^n,$ and any 2-plane $\Pi\subset T_{x}M,$ without loss of generality, we assume that $\Pi$ is spanned by unit orthogonal vectors $Z^{1}$ and $Z^{2}$. Supplementing the basis such that $\{Z^{i}\}_{i=1}^{n}$ is orthonormal basis with respect to $g,$ i.e. $g(Z^{i},Z^{j})=\delta_{ij}.$
Then $\{\frac{1}{V}\frac{\partial}{\partial \theta}, Z^{i}\}$ is orthonormal basis with respect to $h.$\\
The assumption
\begin{equation}
||Rm(h)-\mathbf{K}||_{h}(\theta,x)\leq C e^{-2 \rho(\theta,x)},\nonumber
\end{equation}
implies
\begin{equation}\label{mixterm}
\begin{split}
&|Rm(h)(\frac{1}{V}\frac{\partial}{\partial \theta}, Z^{i},\frac{1}{V}\frac{\partial}{\partial \theta}, Z^{i})-(-1)[h(\frac{1}{V}\frac{\partial}{\partial \theta}, \frac{1}{V}\frac{\partial}{\partial \theta})h(Z^{i}, Z^{i})-h(\frac{1}{V}\frac{\partial}{\partial \theta},Z^{i})^{2}]|\\
&\leq C e^{-2\rho(\theta,x)}
\end{split}
\end{equation}
and
\begin{equation}\label{gah}
|Rm(h)(Z^{i},Z^{j},Z^{i},Z^{j})-(-1)[h(Z^{i},Z^{i})h(Z^{j},Z^{j})-h(Z^{i},Z^{j})^{2}]|\leq C e^{-2\rho(\theta,x)}.
\end{equation}
By direct computation, we have
$$\nabla_{Z}\frac{\partial}{\partial \theta}=V^{-1}\cdot Z(V)\cdot \frac{\partial}{\partial \theta},$$
and $$\nabla_{\frac{\partial}{\partial \theta}}\frac{\partial}{\partial \theta}=-\sum_{i=1}^{n}V\cdot Y^{i}(V)\cdot Y^{i}.$$
As $\vartheta=\frac{1}{V}\frac{\partial}{\partial \theta}$ is a unit normal vector field, for any tangent vector fields $W$ and $Y$ on $M^n,$ the second fundamental form of
$(M^n,g)$ in $(\mathbb{S}^{1}\times M^n, h=V^{2}d\theta^{2}+g)$ is
\begin{equation}
S(W,Y)=h(\nabla_{W}\vartheta, Y)=0.\nonumber
\end{equation}
In other words, $(M^n,g)$ is totally geodesic in $(\mathbb{S}^{1}\times M^n, h).$\\
Together with Gauss-Codazzi equations, we know that (\ref{gah}) implies
\begin{equation}\label{realgah}
\begin{split}
&|Rm(g)(Z^{i},Z^{j},Z^{i},Z^{j})-(-1)[g(Z^{i},Z^{i})g(Z^{j},Z^{j})-g(Z^{i},Z^{j})^{2}]|\leq C e^{-2\rho(\theta,x)}\\
&\leq C e^{-2r(x)},
\end{split}\nonumber
\end{equation}
which means that the sectional curvature of any 2-plane $\Pi\subset T_{x}M,$ for any $x\in M^n$ is approaching $-1,$ hence $(M^n,g)$ is $AH$ of order $2.$

Now Let us take a deep investigation of (\ref{mixterm}). We find that
\begin{equation}\label{ricgtan}
\begin{split}
Ric(g)(Z^{i},Z^{k})&=\sum_{j}Rm(g)(Z^{i},Z^{j},Z^{k},Z^{j})\\
&=\sum_{j}Rm(h)(Z^{i},Z^{j},Z^{k},Z^{j})\\
&=Ric(h)(Z^{i},Z^{k})-Rm(h)(\frac{1}{V}\frac{\partial}{\partial \theta}, Z^{i},\frac{1}{V}\frac{\partial}{\partial \theta}, Z^{k})\\
&=-nh(Z^{i},Z^{k})+O(e^{-2\rho})-Rm(h)(\frac{1}{V}\frac{\partial}{\partial \theta}, Z^{i},\frac{1}{V}\frac{\partial}{\partial \theta}, Z^{k})\\
&=-ng(Z^{i},Z^{k})+O(e^{-2r})-Rm(h)(\frac{1}{V}\frac{\partial}{\partial \theta}, Z^{i},\frac{1}{V}\frac{\partial}{\partial \theta}, Z^{k})\\
\end{split}
\end{equation}
Meanwhile,
\begin{equation}\label{curvhmix}
\begin{split}
&Rm(h)(\frac{1}{V}\frac{\partial}{\partial \theta}, Z^{i},\frac{1}{V}\frac{\partial}{\partial \theta}, Z^{k})\\
=&h(\nabla_{Z^{i}}\nabla_{(\frac{1}{V}\frac{\partial}{\partial \theta})}(\frac{1}{V}\frac{\partial}{\partial \theta})-\nabla_{(\frac{1}{V}\frac{\partial}{\partial \theta})}\nabla_{Z^{i}}(\frac{1}{V}\frac{\partial}{\partial \theta})-\nabla_{[Z^{i},\frac{1}{V}\frac{\partial}{\partial \theta}]}(\frac{1}{V}\frac{\partial}{\partial \theta}),Z^{k})\\
=&-V^{-1}\nabla_{g}^{2}V(Z^{i},Z^{k}).\\
\end{split}
\end{equation}
Substituting (\ref{curvhmix}) to (\ref{ricgtan}), we obtain
\begin{equation}
Ric(g)(Z^{i},Z^{k})+ng(Z^{i},Z^{k})-V^{-1}\nabla_{g}^{2}V(Z^{i},Z^{k})=O(e^{-2r}),\nonumber
\end{equation}
which means
\begin{equation}
||Ric(g)+ng-V^{-1}\nabla_{g}^{2}V||_{g}\leq Ce^{-2r}.\nonumber
\end{equation}
At last, we compute
\begin{equation}
\begin{split}
&(\nabla_{h}Ric(h))(Z,Y,W)\\
=&(\nabla_{Z}^{h}Ric(h))(Y,W)\\
=&\nabla_{Z}^{h}(Ric(h)(Y,W))-Ric(h)(\nabla_{Z}^{h}Y,W)-Ric(h)(Y,\nabla_{Z}^{h}W)\\
=&\nabla_{Z}^{g}(Ric(h)(Y,W))-Ric(h)(\nabla_{Z}^{g}Y,W)-Ric(h)(Y,\nabla_{Z}^{g}W)\\
=&\nabla_{Z}^{g}[Ric(g)(Y,W)-V^{-1}\nabla_{g}^{2}V(Y,W)]\\
-&[Ric(g)(\nabla_{Z}^{g}Y,W)-V^{-1}\nabla_{g}^{2}V(\nabla_{Z}^{g}Y,W)]\\
-&[Ric(g)(Y,\nabla_{Z}^{g}W)-V^{-1}\nabla_{g}^{2}V(Y,\nabla_{Z}^{g}W)]\\
=&\nabla_{g}(Ric(g)-V^{-1}\nabla_{g}^{2}V)(Z,Y,W)\\
\end{split}\nonumber
\end{equation}
and
\begin{equation}
(\nabla_{h}Ric(h))(Z,\frac{1}{V}\frac{\p}{\p \theta},\frac{1}{V}\frac{\p}{\p \theta})=\nabla_{Z}^{g}(V^{-1}\Delta_{g}V),\nonumber
\end{equation}
we therefore conclude
\begin{equation}
||\nabla_{g}(V^{-1}\Delta_{g}V)||_{g}\leq Ce^{-ar},\nonumber
\end{equation}
and
\begin{equation}
||\nabla_{g}(Ric(g)-V^{-1}\nabla_{g}^{2}V)||_{g}\leq Ce^{-ar}.\nonumber
\end{equation}
\end{proof}
We will now deduce from our static flow some basic evolution equations involving some quantities we need to estimate in the following. First we observe that the first equation in the system
\begin{equation}\label{stasicflowg}
\frac{\partial}{\partial t} g= -2Ric(g)-2ng+2V^{-1}\nabla_{g}^{2}V
\end{equation}
\begin{equation}\label{staticflowv}
\frac{\partial}{\partial t} V=\Delta_{g} V-nV
\end{equation}
is only weakly parabolic due to diffeomorphism invariance of the equation. Since the last two terms depend on $g$ and $\p g,$ the principal symbol of the first equation is the same as for the Ricci flow, we could use DeTurck's trick to break the gauge, following the presentation of List \cite{BL} very closely.
\begin{lemm}
Suppose that $(g,V)$ is a solution to
\begin{equation}\label{equationbydiffg}
\frac{\partial}{\partial t} g= -2Ric(g)-2ng+2V^{-1}\nabla_{g}^{2}V+L_{W}g
\end{equation}
\begin{equation}\label{equationbydiffv}
\frac{\partial}{\partial t} V= \Delta_{g} V-nV+dV(W)
\end{equation}
on $M^n\times[0,T],$ and that $W$ is a smooth time dependent vector field on $M^n,$ and that $\varphi_t:M^n\longrightarrow M^n$ is the 1-parameter family of diffeomorphisms generated by $W,$ that is,
\begin{equation}\label{diff}
\left\{
\begin{array}{ll}
    \frac{\partial}{\partial t}\varphi_{t}(x)=W(\varphi_{t}(x),t)  \\
    \varphi_{0}=id.\\
\end{array}
\right.\nonumber
\end{equation}
Then the pullbacks
\begin{equation}\label{change}
\left\{
\begin{array}{ll}
    \bar{g}(t):=(\varphi_{t}^{-1})^{*}g\\
    \bar{V}(t):=(\varphi_{t}^{-1})^{*}V\\
\end{array}
\right.\nonumber
\end{equation}
satisfy (\ref{stasicflowg}) and (\ref{staticflowv}) on $M^n\times[0,T].$ Moreover, $(g,V)(t)$ have the same initial values as $(\bar{g},\bar{V})(t),$ that is,
\begin{equation}
(g,V)(t)=(g_{0},V_{0}).\nonumber
\end{equation}
\end{lemm}
\begin{proof}
Denote by $\{y^{\alpha}\}_{\alpha=1...n}$ the coordinates where $\bar{g}$ and $\bar{V}$ are represented by $\bar{g}_{\alpha\beta}$ and $\bar{V}.$ As the same argument in \cite{Shi} and \cite{BL}, we define new coordinates by $x^{i}:=(y\circ\varphi)^{i}$ for $i=1...n.$ Then
\begin{equation}
\begin{split}
&\varphi_{t}^{\ast}(\bar{V}^{-1}(y,t)\bar{\nabla}_{\bar{g}}^{2}\bar{V}(y,t))_{ij}\\
&=V(x)^{-1}\cdot\varphi_{t}^{\ast}(\bar{\nabla}_{\alpha}\bar{\nabla}_{\beta}\bar{V}dy^{\alpha}\otimes dy^{\beta})_{ij}\\
&=V(x)^{-1}\cdot(\frac{\p x^{i}}{\p y^{\alpha}}\frac{\p x^{j}}{\p y^{\beta}}\nabla_{i}\nabla_{j}V\cdot\frac{\p y^{\alpha}}{\p x^{i}}dx^{i}\frac{\p y^{\beta}}{\p x^{j}}dx^{j})\\
&=V^{-1}\nabla_{g}^{2}V.\\
\end{split}\nonumber
\end{equation}
The rest of the proof is unchanged as \cite{BL} Lemma 3.1.
\end{proof}
An easy computation shows that if $W$ is given by $$W^{k}=g^{ij}(\Gamma_{ij}^{k}-\Gamma_{ij}^{k}(g_{0})),$$ the trace of the tensor which is the difference between the Christoffel symbols of the Levi-Civita connections of $g$ and of $g_{0}$, the equation (\ref{equationbydiffg}) is strictly parabolic. The idea is from DeTurck \cite{DeT}. We pursue this idea via taking all the derivatives with respect to the initial metric $g_{0}.$ This is characterized as follows. For convenience, we use $\hat{g}$ to represent $g_{0},$ and then $\hat{\nabla}$ and $\hat{\Delta}$ are covariant derivative and Laplacian with respect to $g_{0}.$
\begin{lemm}
Let $W^{k}=g^{ij}(\Gamma_{ij}^{k}-\hat{\Gamma}_{ij}^{k}),$ then for given $V,$ the equation (\ref{equationbydiffg}) is strictly parabolic.
\end{lemm}
\begin{proof}
The calculation goes the same as \cite{BL} Lemma 3.2. The equations
(\ref{equationbydiffg}) and (\ref{equationbydiffv}) turn to
\begin{equation}\label{mainequation}
\left\{
\begin{array}{ll}
\begin{split}
    \frac{\partial}{\partial t} g_{ij} &=g^{kl}\hat{\nabla}_{k}\hat{\nabla}_{l}g_{ij}-g^{kl}g_{ip}\hat{g}^{pq}\hat{R}_{jkql}-g^{kl}g_{jp}\hat{g}^{pq}\hat{R}_{ikql}\\
    &+\frac{1}{2}g^{kl}g^{pq}(\hat{\nabla}_{i}g_{pk}\cdot\hat{\nabla}_{j}g_{ql}+2\hat{\nabla}_{k}g_{jp}\cdot\hat{\nabla}_{q}g_{il}-2\hat{\nabla}_{k}g_{jp}\cdot\hat{\nabla}_{l}g_{iq}\\
    &-2\hat{\nabla}_{j}g_{pk}\cdot\hat{\nabla}_{l}g_{iq}-2\hat{\nabla}_{i}g_{pk}\cdot\hat{\nabla}_{l}g_{jq})\\
    &+2V^{-1}\hat{\nabla}_{i}\hat{\nabla}_{j}V-V^{-1}\cdot g^{kl}(\hat{\nabla}_{i}g_{jl}+\hat{\nabla}_{j}g_{il}-\hat{\nabla}_{l}g_{ij})\cdot \hat{\nabla}_{k}V-2ng_{ij}\\
\end{split}\\
    \frac{\partial}{\partial t} V= g^{ij}\hat{\nabla}_{i}\hat{\nabla}_{j}V-nV.\\
    g(\cdot,0)=g_{0}  \\
    V(\cdot,0)=V_{0}.\\
\end{array}
\right.
\end{equation}
It is then obvious to see that these equations are strictly parabolic.
\end{proof}

The above two lemmas show that once we obtain the short-time existence of the initial problem (\ref{mainequation}), after pulling back the solutions to (\ref{mainequation}) by a specific diffeomorphism, then we will also get the short-time existence to our static flow (\ref{originalequation}). Hence now we prove the existence of a solution to the initial value problem (\ref{mainequation}) on $M^n\times[0,T].$ We fix a point $o\in M^n$ and $r(x)$ is the distance function to $o$ with respect to $\hat{g}.$ Since $M^n$ is complete noncompact, it gurantees a family of domains $\{D_{k}\subset M^n: k=1,2,3...\}$ such that for each $k:$
\begin{enumerate}
  \item the boundary $\p D_{k}$ is a smooth $(n-1)$-dimensional submanifold of $M^n;$
  \item the closure $\bar{D}_{k}$ is compact in $M^n;$
  \item $\bigcup_{k=1}^{\infty}D_{k}=M^n;$
  \item $B(o,k)\subset D_{k};$
\end{enumerate}
where $B(o,k)$ is a geodesic ball of center $o$ and radius $k$ with respect to $\hat{g}.$
The parabolic boundary $\Gamma_k$ of $D_k\times[0,T]$ is defined by:
$$\Gamma_k:=(D_k\times\{0\})\cup(\p D_{k}\times[0,T]).$$

For the short time existence of the evolution equation, our strategy is to use Schauder fixed point theorem that
a compact mapping of a closed bounded convex set in a Banach space into itself has a fixed point. Thus we need to construct a Banach space and its closed bounded convex set as well as an operator, and meanwhile show that the operator has a fixed point which is just the solution to the static flow (\ref{mainequation}). Since our initial data $(M^n,\hat{g})$ is asymptotical hyperbolic, which is conformally compactifiable, we can find special coordinate charts called M$\ddot{o}$bius charts. For a concrete definition of M$\ddot{o}$bius charts, please see \cite{Lee} chapter 2. The advantage is that the geometry of $(M^n,\hat{g})$ is uniformly bounded in M$\ddot{o}$bius charts.

Before we introduce our Banach space, we give the definition of general $H\ddot{o}lder$ space used in standard parabolic theory. Denote $Q_{T}:=\Omega\times(0,T).$ For $0<\mu<1$ and $l$ a nonnegative integer, we let
\begin{equation}
\begin{split}
C^{2l+\mu,l+\frac{\mu}{2}}(\bar{Q}_{T})=&\{u; \p^{\beta}\p_{t}^{r}u\in C^{\mu,\frac{\mu}{2}}(\bar{Q}_{T}),\\
&for~ any~ \beta,~ r ~such~ that ~|\beta|+2r\leq 2l \},\\
\end{split}\nonumber
\end{equation}
where $C^{\mu,\frac{\mu}{2}}(\bar{Q}_{T})$ is the set of all functions on $\bar{Q}_{T}$ such that $[u]_{\mu,\frac{\mu}{2};\bar{Q}_{T}}<+\infty,$ endowed with the norm
\begin{equation}
|u|_{\mu,\frac{\mu}{2};\bar{Q}_{T}}=|u|_{0;Q_{T}}+[u]_{\mu,\frac{\mu}{2};\bar{Q}_{T}},\nonumber
\end{equation}
here
\begin{equation}
|u|_{0;\bar{Q}_{T}}=\sup_{(x,t)\in \bar{Q}_{T}}|u(x,t)|\nonumber
\end{equation}
and
\begin{equation}
[u]_{\mu,\frac{\mu}{2};\bar{Q}_{T}}=\sup_{(x,t),(y,s)\in \bar{Q}_{T},(x,t)\neq(y,s)}\frac{|u(x,t)-u(y,s)|}{(|x-y|^{2}+|t-s|)^{\frac{1}{2}}}.\nonumber
\end{equation}
To define $H\ddot{o}lder$ norms for tensor bundle $E,$ let $C^{2l+\mu,l+\frac{\mu}{2}}(M^n\times[0,T];E)$ be the space of tensor fields whose components in each M$\ddot{o}$bius coordinate chart are in $C^{2l+\mu,l+\frac{\mu}{2}}(M^n\times[0,T]).$

We define a space
\begin{equation}
\begin{split}
\mathbf{B}_{\alpha}=&\{s(\cdot,t)\in Sym(T^{*}M\otimes T^{*}M)\mid \\
&e^{2r}s\in C^{\alpha,\frac{\alpha}{2}}(M^n\times[0,T]; T^{*}M\otimes T^{*}M),\\
&e^{2r}\hat{\nabla}s\in C^{\alpha,\frac{\alpha}{2}}(M^n\times[0,T]; T^{*}M\otimes T^{*}M\otimes T^{*}M),t\in[0,T]\}
\end{split}\nonumber
\end{equation}
with norms
\begin{equation}
\begin{split}
\|s\|_{\mathbf{B}_{\alpha}}:&=\|e^{2r}s\|_{C^{\alpha,\frac{\alpha}{2}}(M^n\times[0,T]; T^{*}M\otimes T^{*}M)}\\
&+\|e^{2r}\hat{\nabla}s\|_{C^{\alpha,\frac{\alpha}{2}}(M^n\times[0,T]; T^{*}M\otimes T^{*}M\otimes T^{*}M)}\\
&<+\infty
\end{split}\nonumber
\end{equation}
Obviously, $\mathbf{B}_{\alpha}$ is a Banach space. For any $\epsilon>0,$ we let $$\mathbf{B}_{\alpha,\epsilon}=\{s\in\mathbf{B}_{\alpha}\mid\|s\|_{\mathbf{B}_{\alpha}}\leq\epsilon\}.$$ It is a closed bounded convex set in $\mathbf{B}_{\alpha}$.

Next we need to define our operator $$\mathbf{A}:\mathbf{B}_{\alpha,\epsilon}\longrightarrow \mathbf{B}_{\alpha}$$ and prove that it is well-defined, the range of $\mathbf{A}$ still lies in $\mathbf{B}_{\alpha,\epsilon},$ and $\mathbf{A}$ is compact in $\mathbf{B}_{\alpha,\epsilon}.$ In order to define $\mathbf{A},$ it is necessary to get the estimates of $V.$ We have the following

\begin{lemm}\label{uniquev}
Suppose that $(M^n,\hat{g},\hat{V})$ is $AS$ of order $a\geq2,$ and $\|e^{-r}\hat{V}\|_{C^{2+\alpha}}\leq C.$ Consider
\begin{equation}
\left\{
\begin{array}{ll}
    \frac{\partial}{\partial t} V= g^{ij}\hat{\nabla}_{i}\hat{\nabla}_{j}V-nV  \\
    V(\cdot,0)=\hat{V}>0~~~ on~~~M^{n}\times[0,T]\\
\end{array}
\right.\nonumber
\end{equation}
If $g(x,t)\in\mathbf{B}_{\alpha,\epsilon}+\hat{g},$ then the above initial problem has a unique positive solution $V$ from the class $C^{2+\alpha,1+\frac{\alpha}{2}}(M^n\times[0,T])$ and it satisfies
\begin{equation}
e^{r}(V-\hat{V})\in C^{2+\alpha,1+\frac{\alpha}{2}}(M^n\times[0,T]).\nonumber
\end{equation}
\end{lemm}
\begin{proof}
For any $x\in M^n,$ there exists $k$ such that $x\in B(x,1)\subset B(x,2)\subset D_{k}.$
We solve the following initial-boundary value problem
\begin{equation}\label{inibouv}
\left\{
\begin{array}{ll}
    \frac{\partial}{\partial t} V= g^{ij}\hat{\nabla}_{i}\hat{\nabla}_{j}V-nV,~~~on~ D_{k}\times[0,T]\\
    V\mid_{\Gamma_k}=\hat{V}.\\
\end{array}
\right.
\end{equation}
If we let $v=V-\hat{V},$ (\ref{inibouv}) is equivalent to the following initial-boundary value problem
\begin{equation}\label{v}
\left\{
\begin{array}{ll}
    \frac{\p}{\p t}v=g^{ij}\hat{\nabla}_{i}\hat{\nabla}_{j}v-nv+g^{ij}\hat{\nabla}_{i}\hat{\nabla}_{j}\hat{V}-n\hat{V},~~~on~ D_{k}\times[0,T]\\
    v\mid_{\Gamma_k}=0.\\
\end{array}
\right.
\end{equation}
From the assumption that $g(x,t)=s(x,t)+\hat{g}\in\mathbf{B}_{\alpha,\epsilon}+\hat{g},$ note that we choose $\epsilon\ll 1$ such that the above equation is uniformly parabolic, together with standard parabolic theory(see \cite{Lady} P320 Theorem 5.2), we know linear initial-boundary value problem (\ref{v}) has a unique solution $v_{k}$ from the class $C^{2+\alpha,1+\frac{\alpha}{2}}(\bar{D}_{k}\times[0,T]),$ and it satisfies
\begin{equation}\label{VkminusVintial}
 \|v_{k}\|_{C^{2+\alpha,1+\frac{\alpha}{2}}(B(x,1)\times[0,T])}\leq C \| g^{ij}\hat{\nabla}_{i}\hat{\nabla}_{j}\hat{V}-n\hat{V}\|_{C^{\alpha,\frac{\alpha}{2}}(B(x,2)\times[0,T])}.
\end{equation}
Since the right hand side in (\ref{VkminusVintial}) is independent of $k$ and we have $\bigcup_{k=1}^{\infty}D_{k}=M^n,$ we can take the limit $k\longrightarrow\infty$ and get the convergence of a subsequence of the solutions $v_{k}$ in the $C^{2+\beta,1+\frac{\beta}{2}}(\beta<\alpha)$ topology on compact subsets of $M^n\times [0,T]$ to a $C^{2+\alpha,1+\frac{\alpha}{2}}$ solution $v$ on $M^n\times [0,T]$ by the theorem of Arzela-Ascoli for any given $s(x,t)\in\mathbf{B}_{\alpha,\epsilon}.$
From (\ref{VkminusVintial}), we obtain
\begin{equation}\label{vestimate}
\begin{split}
&\|v\|_{C^{2+\alpha,1+\frac{\alpha}{2}}(B(x,1)\times[0,T])}\\
\leq&C(\| g-\hat{g}\|_{C^{\alpha,\frac{\alpha}{2}}(B(x,2)\times[0,T])}\cdot\|\hat{V}\|_{C^{2+\alpha}(B(x,2))}+\|\hat{\Delta}\hat{V}-n\hat{V}\|_{C^{\alpha}(B(x,2))}).\\
\end{split}\nonumber
\end{equation}
When $y\in B(x,2),$ $r(x)$ and $r(y)$ is equivalent, thus
\begin{equation}
\|v\|_{C^{2+\alpha,1+\frac{\alpha}{2}}(B(x,1)\times[0,T])}\leq Ce^{-r(x)},~~~t\in[0,T].\nonumber
\end{equation}
Hence $V=v+\hat{V}$ is a solution to (\ref{inibouv}). In order to get that $V$ is positive and unique, it is enough to get the same results for the original equation (\ref{staticflowv}).
As $\hat{V}>0$ and $\hat{V}=O(e^{r}),$ though the manifold is complete noncompact, there still exists a positive constant $d$ such that $\hat{V}\geq d$ on the whole $M^n.$ We define $J(x,t)=d-e^{nt}V(x,t),$ then $J$ satisfies
\begin{equation}
\frac{\partial}{\partial t} J= \Delta_{g}J,~~on~M^n\times[0,T]\nonumber
\end{equation}
and $J(\cdot,0)\leq0.$
Since
\begin{equation}
\begin{split}
&\int_{0}^{T}\int_{M^{n}}\exp(-r(x)^{2})J_{+}^{2}(x,t)d\mu_{g(t)}dt\\
\leq& C\int_{0}^{T}\int_{M^{n}}\exp(-r(x)^{2}+2r(x)+2nt)d\mu_{\hat{g}}dt\\
<&\infty,
\end{split}\nonumber
\end{equation}
due to the maximum principle of Karp and Li \cite{KL}, we get $J\leq 0$ on $M^n\times[0,T],$ which means $V$ is always positive on $M^n\times[0,T].$\\
If $V_{1}$ and $V_{2}$ are two solutions to (\ref{staticflowv}) with the same initial data $\hat{V}.$ Let $I(x,t)=e^{nt}(V_{1}-V_{2}),$ then
\begin{equation}
\frac{\partial}{\partial t} I= \Delta_{g}I,~~on~M^n\times[0,T]\nonumber
\end{equation}
and $J(\cdot,0)=0.$ Also by the same maximum principle, we conclude $I=0$ on $M^n\times[0,T],$ that is, $V_{1}=V_{2}$ on $M^n\times[0,T].$
\end{proof}

Motivated by the above lemma, we define the operator $\mathbf{A}:\mathbf{B}_{\alpha,\epsilon}\longrightarrow \mathbf{B}_{\alpha}$ by
\begin{equation}
\mathbf{A}(s)= G-\hat{g}\nonumber
\end{equation}
where
\begin{equation}
\left\{
\begin{array}{ll}
    \frac{\partial}{\partial t} G_{ij}=g^{kl}\hat{\nabla}_{k}\hat{\nabla}_{l}G_{ij}+H_{ij}(\hat{R}m,\hat{\nabla}g,g,V^{-1}\nabla_{g}^{2}V),~~~on~ M^n\times[0,T]\\
    G_{ij}(\cdot,0)=\hat{g}_{ij}.\\
\end{array}
\right.\nonumber
\end{equation}
Here $H_{ij}(\hat{R}m,\hat{\nabla}g,g,V^{-1}\nabla_{g}^{2}V)$ contains all remaining terms of the evolution equation of $g=s+\hat{g}$ in (\ref{mainequation}) and $V$ is the unique positive solution in Lemma \ref{uniquev} for given $g=s+\hat{g}$. Since we first need to get the a priori estimates, we claim the following:
\begin{prop}\label{mainprop}
Suppose $(M^n,\hat{g})$ is a complete noncompact AH Riemannian manifold of order $\eta$ and it contains an essential set. Let $f,$ which is smooth on $M^{n}\times(0, T]$ and continuous
on $M^{n}\times[0, T]$ satisfy
\begin{equation}
(\frac{\p}{\p t}-\hat{\Delta})f(x,t)\leq Q(x,t)\nonumber
\end{equation}
on $M^n \times [0,T],$ where $Q(x,t)$ is a function on $M^n\times[0,T]$ satisfying $|Q(x,t)|\leq Le^{-\mu r(x)}.$ Here $\mu\in\mathbb{R}$ and $L>0$ are constants and $r(x)$ is the distance function to the essential set with respect to $\hat{g}.$ Assume that $|f|+|\hat{\nabla}f|\leq C_{1}$ on $M^n\times[0,T]$. If $f(x,0)\leq 0,$ then there exists $\tilde{T}=\tilde{T}(\mu,n,C)\leq T,$ such that for all $(x,t)\in M^n \times [0,\tilde{T}],$ we have $$f(x,t)\leq 2Lte^{-\mu r(x)}.$$
\end{prop}
\begin{proof}
We will now adapt a maximum principle for heat equations by Ecker and Huisken, refer to Theorem 4.3 in \cite{EH}.\\
Let $S(x,t)=e^{\mu r(x)}f(x,t)-2Lt.$ As $r(x)$ is the distance function to the essential set, we extend $r(x)$ smoothly to the interior of the essential set, then $S(x,t)$ is smooth on $M^{n}\times(0, T]$ and continuous
on $M^{n}\times[0, T].$ We have $S(x,0)\leq 0$ and
\begin{equation}
\begin{split}
\frac{\p S}{\p t}&\leq\hat{\Delta}S-2\mu\hat{\nabla}r\cdot\hat{\nabla}S+(\mu^{2}|\hat{\nabla}r|^{2}-\mu\hat{\Delta}r)S\\
&+2(\mu^{2}|\hat{\nabla}r|^{2}-\mu\hat{\Delta}r)Lt+e^{\mu r}Q-2L.\\
\end{split}\nonumber
\end{equation}
First notice that in view of the assumption $(M^n,\hat{g})$ is $AH$ of order $\eta$ and lemma 2.1 in \cite{HQS}, $|\hat{\Delta}r-n|\leq Ce^{-2r}$ for $\eta>2,$ $Cre^{-2r}$ for $\eta=2,$ and $Ce^{-\eta r}$ for $0<\eta<2,$ where the constant $C$ only depends on the $AH$ constant of $(M^n,\hat{g}).$\\ Now observe that $|\hat{\nabla}r|\equiv1,$ we can therefore choose $T_{1}\leq T$ such that for $t\leq T_{1},$ we have $$2(\mu^{2}|\hat{\nabla}r|^{2}-\mu\hat{\Delta}r)Lt\leq L.$$
Using the assumption $|Q(x,t)|\leq Le^{-\mu r(x)}$ we obtain $$e^{\mu r}Q-L\leq 0.$$
We therefore conclude for $t\leq T_1,$
\begin{equation}
\begin{split}
\frac{\p S}{\p t}-\hat{\Delta}S\leq\vec{\Lambda}\cdot\hat{\nabla}S+\lambda S,
\end{split}\nonumber
\end{equation}
where the vector $\vec{\Lambda}=-2\mu\hat{\nabla}r,$ $\lambda=\mu^{2}|\hat{\nabla}r|^{2}-\mu\hat{\Delta}r$ and the inner product is with respect to $\hat{g}.$\\
We then proceed to show all the conditions in \cite{EH} Theorem 4.3 are satisfied in our situations and we could make the finial conclusion.\\
Since $(M^n,\hat{g})$ is $AH$, there exists a constant $k,$ say $-(C+1),$ such that $Ric_{\hat{g}}\geq (n-1)k.$ Applying now the absolute volume comparison result we arrive at
$$vol_{\hat{g}}B(x,r)\leq v(n,k,r),$$
where $v(n,k,r)$ denotes the volume of a ball of radius $r$ in the constant-curvature space form $S_{k}^{n}.$ For $k=-(C+1),$ we have
\begin{equation}
\begin{split}
v(n,k,r)&=\int_{0}^{r}\int_{S^{n-1}}(\frac{\sinh\sqrt{C+1}\tau}{\sqrt{C+1}})^{n-1}d\tau d\theta\\
&\leq \exp(Dr)\\
&\leq \exp(\frac{D}{2}(1+r^{2}))\\
\end{split}\nonumber
\end{equation}
where the constant $D$ only depends on $C$ and $n.$ Therefore
$$vol_{\hat{g}}B(x,r)\leq \exp(\frac{D}{2}(1+r^{2})).$$
We also estimate
$$\sup_{M^{n}\times[0,T_{1}]}|\lambda|\leq \mu^{2}+|\mu|(n+C)<\infty$$
and $$\sup_{M^{n}\times[0,T_{1}]}|\vec{\Lambda}|\leq 2|\mu|,$$
as well as $$\sup_{M^{n}\times[0,T_{1}]}|\frac{d}{dt}\hat{g}_{ij}|=0.$$
Now note that
\begin{equation}
\int_{0}^{T_{1}}\int_{M^{n}}\exp(-r(x)^{2})|\hat{\nabla}S|^{2}(x)d\mu_{\hat{g}}dt<\infty,\nonumber
\end{equation}
since $|\hat{\nabla}S|=|\mu e^{\mu r}f\hat{\nabla}r+e^{\mu r}\hat{\nabla}f|\leq C(C_{1},\mu)e^{\mu r}.$
From Theorem 4.3 in \cite{EH} we obtain finally
$S\leq0$ on $M^{n}\times[0,T_{1}].$ Therefore, if we let $\tilde{T}=T_{1},$ we have
$$f(x,t)\leq 2Lte^{-\mu r(x)},$$ for all $(x,t)\in M \times [0,\tilde{T}].$
\end{proof}

\begin{lemm}\label{finallemma}
Let $g(x,t)=s(x,t)+\hat{g}\in\mathbf{B}_{\alpha,\epsilon}+\hat{g},$ and $V$ is the unique positive solution we get in Lemma \ref{uniquev}. Then the initial problem
\begin{equation}\label{uniqueG}
\left\{
\begin{array}{ll}
    \frac{\partial}{\partial t} G_{ij}=g^{kl}\hat{\nabla}_{k}\hat{\nabla}_{l}G_{ij}+H_{ij}(\hat{R}m,\hat{\nabla}g,g,V^{-1}\nabla_{g}^{2}V),~~~on~ M^n\times[0,T]\\
    G_{ij}(\cdot,0)=\hat{g}_{ij}.\\
\end{array}
\right.
\end{equation}
has a solution and
\begin{equation}
e^{2r}(G-\hat{g})\in C^{2+\alpha,1+\frac{\alpha}{2}}(M^n\times[0,T]).\nonumber
\end{equation}
Moreover, there exists $T_{0}=T_{0}(n,C,\epsilon)\in(0,T]$ such that $G\in\mathbf{B}_{\alpha,\epsilon}+\hat{g}$ for $t\in[0,T_{0}].$
\end{lemm}
\begin{proof}
For any $x\in M^n,$ there exists $k$ such that $x\in B(x,1)\subset B(x,2)\subset D_{k}.$
We solve the following linear initial-boundary value problem
\begin{equation}\label{inibouG}
\left\{
\begin{array}{ll}
    \frac{\partial}{\partial t} G_{ij}=g^{kl}\hat{\nabla}_{k}\hat{\nabla}_{l}G_{ij}+H_{ij}(\hat{R}m,\hat{\nabla}g,g,V^{-1}\nabla_{g}^{2}V),~~~on~ D_{k}\times[0,T]\\
    G_{ij}\mid_{\Gamma_k}=\hat{g}_{ij}\\
\end{array}
\right.
\end{equation}
where
\begin{equation}\label{Hij}
\begin{split}
H_{ij}&=-g^{kl}g_{ip}\hat{g}^{pq}\hat{R}_{jkql}-g^{kl}g_{jp}\hat{g}^{pq}\hat{R}_{ikql}+g^{kl}g^{pq}(\frac{1}{2}\hat{\nabla}_{i}g_{pk}\cdot\hat{\nabla}_{j}g_{ql}+\\
    &\hat{\nabla}_{k}g_{jp}\cdot\hat{\nabla}_{q}g_{il}-\hat{\nabla}_{k}g_{jp}\cdot\hat{\nabla}_{l}g_{iq}-\hat{\nabla}_{j}g_{pk}\cdot\hat{\nabla}_{l}g_{iq}-\hat{\nabla}_{i}g_{pk}\cdot\hat{\nabla}_{l}g_{jq})\\
    &-g^{kl}(\hat{\nabla}_{i}g_{jl}+\hat{\nabla}_{j}g_{il}-\hat{\nabla}_{l}g_{ij})\cdot V^{-1}\hat{\nabla}_{k}V+2V^{-1}\hat{\nabla}_{i}\hat{\nabla}_{j}V-2ng_{ij}.\\
\end{split}\nonumber
\end{equation}
Correspondingly, $q=G-\hat{g}$ satisfies
\begin{equation}\label{q}
\left\{
\begin{array}{ll}
    \frac{\partial}{\partial t} q_{ij}=g^{kl}\hat{\nabla}_{k}\hat{\nabla}_{l}q_{ij}+H_{ij}(\hat{R}m,\hat{\nabla}g,g,V^{-1}\nabla_{g}^{2}V),~~~on~ D_{k}\times[0,T]\\
    q_{ij}\mid_{\Gamma_k}=0.\\
\end{array}
\right.\nonumber
\end{equation}
We rewrite $H$ and find that $H$ is a combination of the following terms:\\
$\hat{R}ic+n\hat{g}-\hat{V}^{-1}(\hat{\nabla}^{2}\hat{V}),$ $s,$ $s*\hat{R}m,$ $s*s*\hat{R}m,$
$V^{-1}\hat{\nabla}s*\hat{\nabla}V,$ $V^{-1}*s*\hat{\nabla}s*\hat{\nabla}V,$ $s*\hat{\nabla}s*\hat{\nabla}s,$
$s*s*\hat{\nabla}s*\hat{\nabla}s,$ $(V\hat{V})^{-1}(V-\hat{V})\hat{\nabla}^{2}\hat{V},$ and $V^{-1}\hat{\nabla}^{2}(V-\hat{V}).$\\
By the assumption that $s\in\mathbf{B}_{\alpha,\epsilon}$ and $V$ comes from the unique positive solution in Lemma \ref{uniquev}, we know that $g^{kl}$ and $H_{ij}$ belong to the class $C^{\alpha,\frac{\alpha}{2}}(\bar{D}_{k}\times[0,T]),$ from the standard parabolic theory, we know linear initial-boundary value problem (\ref{inibouG}) has a unique solution $G^{k}_{ij}$ from the class $C^{2+\alpha,1+\frac{\alpha}{2}}(\bar{D}_{k}\times[0,T]),$ and it satisfies
\begin{equation}\label{Gkestimate}
\|G^{k}_{ij}\|_{C^{2+\alpha,1+\frac{\alpha}{2}}(D_{k}\times[0,T])}\leq C(\|\hat{g}\|_{C^{2+\alpha}(D_{k})}+\|H_{ij}\|_{C^{\alpha,\frac{\alpha}{2}}(D_{k}\times[0,T])})\leq C
\end{equation}
and
\begin{equation}\label{Gminushatgestimate}
 \|G^{k}_{ij}-\hat{g}_{ij}\|_{C^{2+\alpha,1+\frac{\alpha}{2}}(B(x,1)\times[0,T])}\leq C(\|H_{ij}\|_{C^{\alpha,\frac{\alpha}{2}}(B(x,2)\times[0,T])}+|G^{k}_{ij}-\hat{g}_{ij}|_{\Gamma_{k}})
\end{equation}
Like what we did above, as the right hand side of (\ref{Gkestimate}) is independent of $k,$ we can take the limit $k\longrightarrow\infty$ and get the convergence of a subsequence of the solutions $G_{ij}^{k}$ in the $C^{2+\beta,1+\frac{\beta}{2}}(\beta<\alpha)$ topology on compact subsets of $M^n\times [0,T]$ to a $C^{2+\alpha,1+\frac{\alpha}{2}}$ solution $G_{ij}$ on $M^n\times [0,T]$ by the theorem of Arzela-Ascoli for given $s(x,t)\in\mathbf{B}_{\alpha,\epsilon}.$\\
From (\ref{Gminushatgestimate}), we obtain
\begin{equation}\label{Gminushatg}
 \|G_{ij}-\hat{g}_{ij}\|_{C^{2+\alpha,1+\frac{\alpha}{2}}(B(x,1)\times[0,T])}\leq C\|H_{ij}\|_{C^{\alpha,\frac{\alpha}{2}}(B(x,2)\times[0,T])}\nonumber
\end{equation}
When $y\in B(x,2),$ $r(x)$ and $r(y)$ is equivalent, thus
\begin{equation}\label{g2alpha}
\|G-\hat{g}\|_{C^{2+\alpha,1+\frac{\alpha}{2}}(B(x,1)\times[0,T])}\leq L_{1}e^{-2r(x)},~~~t\in[0,T].
\end{equation}
Let $f=\|G-\hat{g}\|_{\hat{g}},$ by direct computation, $f$ satisfies
\begin{equation}\label{f}
\begin{split}
\frac{\p f}{\p t}&\leq g^{kl}\hat{\nabla}_{k}\hat{\nabla}_{l}f+F\\
&\leq \hat{\Delta}f+(g^{kl}-\hat{g}^{kl})\hat{\nabla}_{k}\hat{\nabla}_{l}f+F\\
&= \hat{\Delta}f+Q
\end{split}\nonumber
\end{equation}
where
\begin{equation}
\begin{split}
Q(x,t) &\leq C (\|H\|_{0;B(x,1)\times[0,T]}+\|s\|_{0;B(x,1)\times[0,T]}\cdot\|\nabla^{2}(G-\hat{g})\|_{0;B(x,1)\times[0,T]})\\
&\leq Ce^{-2r(x)}.\\
\end{split}\nonumber
\end{equation}
on $M^n\times[0,T].$
By applying Proposition \ref{mainprop}, there exists $T_1\in(0,T]$ which only depends on $n,$ $a,$ and $AS$ constant such that
\begin{equation}\label{Gmax}
\|G-\hat{g}\|(x,t)\leq L_{2}te^{-2r(x)}.
\end{equation}
According to Proposition \ref{mainprop}, if we choose $T$ to be small at the beginning, $T_{1}$ could reach $T.$ Without loss of generality, we may let $T_1=T.$\\
Now observe that once we have (\ref{g2alpha}) and (\ref{Gmax}), using Interpolation Inequality(see \cite{Lady} P80 Lemma3.2), we obtain that for any $\delta\in(0,\sqrt{T}],$ there exist constants $c_1,$ $c_2,$ $d_1,$ and $d_2,$ which only depend on $\alpha$ and $n,$ such that
\begin{equation}\label{alpha}
\begin{split}
&\|G-\hat{g}\|_{C^{\alpha,\frac{\alpha}{2}}(B(x,1)\times(0,T])}\\
\leq&c_1\delta^2 \|G-\hat{g}\|_{C^{2+\alpha,1+\frac{\alpha}{2}}(B(x,1)\times(0,T])}+c_2\delta^{-\alpha}\|G-\hat{g}\|_{0;B(x,1)\times(0,T]},\\
\end{split}
\end{equation}
and
\begin{equation}\label{oneplusalpha}
\begin{split}
&\|G-\hat{g}\|_{C^{1+\alpha,\frac{1+\alpha}{2}}(B(x,1)\times(0,T])}\\
\leq&d_{1}\delta \|G-\hat{g}\|_{C^{2+\alpha,1+\frac{\alpha}{2}}(B(x,1)\times(0,T])}+d_{2}\delta^{-(1+\alpha)}\|G-\hat{g}\|_{0;B(x,1)\times(0,T]}.\\
\end{split}
\end{equation}
Let $\delta=\sqrt{T},$ then we can always choose $T=T_{0}$ such that both (\ref{alpha}) and (\ref{oneplusalpha}) are no greater than $\frac{\epsilon}{2}e^{-2r},$ which means
\begin{equation}
\begin{split}
\|G-\hat{g}\|_{\mathbf{B}_{\alpha}}&=\|e^{2r}(G-\hat{g})\|_{C^{\alpha,\frac{\alpha}{2}}(M^n\times[0,T_{0}]; T^{*}M\otimes T^{*}M)}\\
&+\|e^{2r}\hat{\nabla}(G-\hat{g})\|_{C^{\alpha,\frac{\alpha}{2}}(M^n\times[0,T_{0}]; T^{*}M\otimes T^{*}M)}\\
&\leq \frac{\epsilon}{2}+\frac{\epsilon}{2}=\epsilon.
\end{split}\nonumber
\end{equation}
We therefore conclude $G\in\mathbf{B}_{\alpha,\epsilon}+\hat{g}$ when $t\in[0,T_{0}].$
\end{proof}
\begin{lemm}\label{weiyixing}
Let $G$ be what we obtained in the above lemma. Then $G$ is unique in $\mathbf{B}_{\alpha,\epsilon}+\hat{g}$ for $t\in [0,T_{0}].$
\end{lemm}
\begin{proof}
Suppose that $G_1$ and $G_{2}$ are two solutions to the initial problem (\ref{uniqueG}), and that both $G_1$ and $G_{2}$ satisfy
\begin{equation}
e^{2r}(G_{i}-\hat{g})\in C^{2+\alpha,1+\frac{\alpha}{2}}(M^n\times[0,T])\nonumber
\end{equation}
and $G_{i}\in\mathbf{B}_{\alpha,\epsilon}+\hat{g}$ for $t\in[0,T_{0}],$ $i=1,2.$\\
Let $u(x,t):=\|G_{1}-G_{2}\|_{\hat{g}}(x,t),$ then $u$ satisfies
$$\frac{\p u}{\p t}\leq g^{kl}\hat{\nabla}_{k}\hat{\nabla}_{l}u$$ with $u(x,0)=0.$ Define $h(x,t):=-\frac{r(x)^2}{4(2\sigma-t)},$ where $\sigma$ is an arbitrary time in $[0,T_{0}]$. Note that $h$ is a locally Lipschitz function defined on $M^{n}\times[0,2\sigma).$ Since $|\hat{\nabla}r|_{\hat{g}}=1,$ it follows that
\begin{equation}\label{heatkernal}
|\hat{\nabla}h|^{2}_{\hat{g}}+\frac{\p h}{\p t}=0,~~~a.e.~~
\end{equation}
Let $0\leq\varphi_{s}\leq1$ be a cutoff function which is $1$ inside $B(o,s)$ and compactly supported in $B(o,s+1)$ with $|\hat{\nabla}\varphi_{s}|_{\hat{g}}\leq2.$ Multiplying the inequality $\frac{\p u}{\p t}-g^{kl}\hat{\nabla}_{k}\hat{\nabla}_{l}u\leq0$ by the compactly supported Lipschitz function $\varphi_{s}^{2}e^{h}u$ and integrating by parts, we have
\begin{equation}
\begin{split}
0\geq&\int_{0}^{\sigma}\int_{M^n}e^{h}(\varphi_{s}^{2}|\hat{\nabla}u|^{2}_{g}+2\langle\hat{\nabla}\varphi_{s},\hat{\nabla}u\rangle_{g}\varphi_{s}u+\varphi_{s}^{2}u\langle\hat{\nabla}h,\hat{\nabla}u\rangle_{g})d\mu_{\hat{g}}dt\\
 & +\int_{0}^{\sigma}\int_{M^n}e^{h}\varphi_{s}^{2}u(\hat{\nabla}_{k}g^{kl}\cdot\hat{\nabla}_{l}u)d\mu_{\hat{g}}dt+\frac{1}{2}\int_{0}^{\sigma}\int_{M^n}e^{h}\varphi_{s}^{2}(\frac{\p}{\p t}u^2)d\mu_{\hat{g}}dt\\
 \geq&\int_{0}^{\sigma}\int_{M^n}e^{h}(\varphi_{s}^{2}|\hat{\nabla}u|^{2}_{g}-\frac{1}{\lambda^2}u^2|\hat{\nabla}\varphi_{s}|_{g}^{2}-\lambda^{2}\varphi_{s}^{2}|\hat{\nabla}u|_{g}^{2})d\mu_{\hat{g}}dt\\
 &+\int_{0}^{\sigma}\int_{M^n}e^{h}(-\frac{1}{2}\varphi_{s}^{2}u^{2}|\hat{\nabla}h|_{g}^{2}-\frac{1}{2}\varphi_{s}^{2}|\hat{\nabla}u|_{g}^{2})d\mu_{\hat{g}}dt\\
 &+\int_{0}^{\sigma}\int_{M^n}e^{h}\varphi_{s}^{2}u(\hat{\nabla}_{k}g^{kl}\cdot\hat{\nabla}_{l}u)d\mu_{\hat{g}}dt\\
 &+\frac{1}{2}\int_{M^n}e^{h}\varphi_{s}^{2}u^{2}d\mu_{\hat{g}}\mid_{0}^{\sigma}+\frac{1}{2}\int_{0}^{\sigma}\int_{M^n}e^{h}\varphi_{s}^{2}u^{2}|\hat{\nabla}h|^{2}_{\hat{g}}d\mu_{\hat{g}}dt.\\
\end{split}\nonumber
\end{equation}
Here we used the Cauchy-Schwarz inequality and the equality (\ref{heatkernal}). Then
\begin{equation}
\begin{split}
&(\int_{M^n}e^{h}\varphi_{s}^{2}u^{2}d\mu_{\hat{g}})(\sigma)\leq\frac{2}{\lambda^2}\int_{0}^{\sigma}\int_{M^n}e^{h}u^2|\hat{\nabla}\varphi_{s}|_{g}^{2}d\mu_{\hat{g}}dt\\ &+\int_{0}^{\sigma}\int_{M^n}e^{h}\varphi_{s}^{2}\{|\hat{\nabla}u|^{2}_{g}(2\lambda^{2}-1)-2u(\hat{\nabla}_{k}g^{kl}\cdot\hat{\nabla}_{l}u)+u^{2}(|\hat{\nabla}h|^{2}_{g}-|\hat{\nabla}h|^{2}_{\hat{g}})\}d\mu_{\hat{g}}dt.\\
\end{split}\nonumber
\end{equation}
By the assumption and all the estimates in the above lemma, we know that $g-\hat{g}\in\mathbf{B}_{\alpha,\epsilon}$ and $G_{1},G_{2}\in\mathbf{B}_{\alpha,\epsilon}+\hat{g}.$ Then
\begin{equation}
u=\|G_{1}-G_{2}\|_{\hat{g}}\leq \|G_{1}-\hat{g}\|_{\hat{g}}+\|G_{2}-\hat{g}\|_{\hat{g}}\leq 2\epsilon e^{-2r},\nonumber
\end{equation}
Due to Kato's inequality, for a smooth tensor field $W$ with compact support, $\mid\nabla|W|\mid\leq\mid\nabla W\mid,$ we get
\begin{equation}
\|\hat{\nabla}u\|_{\hat{g}}=\|\hat{\nabla}\|G_{1}-G_{2}\|_{\hat{g}}\|_{\hat{g}}\leq\|\hat{\nabla}(G_{1}-G_{2})\|_{\hat{g}}\leq 2\epsilon e^{-2r}.\nonumber
\end{equation}
Hence
\begin{equation}
|\hat{\nabla}u|_{g}\leq\|g-\hat{g}\|_{\hat{g}}\cdot\|\hat{\nabla}u\|_{\hat{g}}+\|\hat{\nabla}u\|_{\hat{g}}\leq 4\epsilon e^{-2r}.\nonumber
\end{equation}
As
\begin{equation}
|u^{2}(|\hat{\nabla}h|^{2}_{g}-|\hat{\nabla}h|^{2}_{\hat{g}})|\leq u^{2}\cdot\|g-\hat{g}\|_{\hat{g}}\cdot|\hat{\nabla}h|^{2}_{\hat{g}},\nonumber
\end{equation}
and the integration domain of time is from $0$ to $\sigma,$ so
\begin{equation}
|u^{2}(|\hat{\nabla}h|^{2}_{g}-|\hat{\nabla}h|^{2}_{\hat{g}})|\leq \frac{1}{(2\sigma-t)^2}\epsilon^{3}r^{2}e^{-6r}\leq\frac{\epsilon^{3}}{\sigma^{2}}e^{-4r}.\nonumber
\end{equation}
Since $\epsilon$ is small, we can always choose $\lambda$ such that
\begin{equation}
\begin{split}
&|\hat{\nabla}u|^{2}_{g}(2\lambda^{2}-1)-2u(\hat{\nabla}_{k}g^{kl}\cdot\hat{\nabla}_{l}u)+u^{2}(|\hat{\nabla}h|^{2}_{g}-|\hat{\nabla}h|^{2}_{\hat{g}})\}\\
\leq &\{16(2\lambda^{2}-1)+8\epsilon+\frac{\epsilon}{\sigma^{2}}\}\epsilon^{2}e^{-4r}\\
\leq &0.
\end{split}\nonumber
\end{equation}
Then
\begin{equation}
\begin{split}
(\int_{M^n}e^{h}\varphi_{s}^{2}u^{2}d\mu_{\hat{g}})(\sigma)&\leq\frac{2}{\lambda^2}\int_{0}^{\sigma}\int_{M^n}e^{h}u^2|\hat{\nabla}\varphi_{s}|_{g}^{2}d\mu_{\hat{g}}dt\\ &\leq \frac{32\epsilon^{3}}{\lambda^2}\int_{0}^{\sigma}\int_{B(o,s+1)\setminus B(o,s)}exp\{-\frac{1}{8\sigma}r^{2}-6r\}d\mu_{\hat{g}}dt.\\
\end{split}\nonumber
\end{equation}
The right hand side tends to zero as $s\rightarrow\infty,$ and we conclude that $u\equiv0$ on $M^n\times[0,\sigma].$ Since $\sigma$ is arbitrary in $[0,T_{0}],$ we conclude that $u\equiv0$ on $M^n\times[0,T_{0}].$ Thus the uniqueness is obtained.
\end{proof}
Finally, we arrive at the proof of our main result.
\begin{proof}[Proof of Theorem \ref{mainresult}]
For $t\in [0,T_{0}],$ and for any $s\in\mathbf{B}_{\alpha,\epsilon},$ let $g=\hat{g}+s,$ by Lemma \ref{uniquev}, we get a unique positive $V$ with good estimates. Plugging $g$ and $V$ in equation (\ref{uniqueG}), from Lemma \ref{finallemma}, we get a solution $G.$ Now we find that the operator $$\mathbf{A}: \mathbf{B}_{\alpha,\epsilon}\longrightarrow \mathbf{B}_{\alpha}$$
$$\mathbf{A}(s)= G-\hat{g}$$
has the range in $\mathbf{B}_{\alpha,\epsilon}$ by Lemma \ref{finallemma} and is well-defined in $\mathbf{B}_{\alpha,\epsilon}$ by Lemma \ref{weiyixing}. Lemma \ref{finallemma} also shows that $\mathbf{A}$ is compact in $\mathbf{B}_{\alpha,\epsilon}.$ Then Schauder fixed point theorem promises us a fixed point $s_{0}\in\mathbf{B}_{\alpha,\epsilon}$ such that $$\mathbf{A}(s_{0})=s_{0}.$$ Let $g=s_{0}+\hat{g},$ and we denote the unique solution from Lemma \ref{uniquev} by $V.$ Obviously, $(g,V)$ are solutions to our static flow (\ref{mainequation}) in $C^{2+\alpha,1+\frac{\alpha}{2}}(M^n\times[0,T_0])$. Then we can improve the spatial regularity step by step, and by bootstrapping and the equation we can improve the regularity in time as well, hence we get a smooth solution. Since $(g,V)$ is a pullback of the solution to the original static flow (\ref{originalequation}) by a specific diffeomorphism, we get the short-time existence of the static flow (\ref{originalequation}).
\end{proof}

\section{Asymptotic expansions at conformal infinity}
In this section, we mainly prove Theorem \ref{expansion}. Suppose the triple $(M^n,g,V)$ is as that of Theorem \ref{expansion}. We also assume that the metric $g$ is sufficiently regular at infinity, which means that it has asymptotic expansions to high enough order, in general involving log terms. We follow the methods used in \cite{G}.
\begin{lemm}\label{uniquelymnobigthann}
Suppose that $(M^n,g,V),$ $\hat{g}$ and $\tau$ are as that of Theorem \ref{expansion}. Then $\partial_{\tau}^{m}g_{\tau}|_{\tau=0}$ and $\partial_{\tau}^{(m-1)}V|_{\tau=0}$ can uniquely determine by $\hat{g}$ for any
$1\leq m\leq n-1.$
\end{lemm}
\begin{proof}
By the assumption, and $g_{\tau}|_{\tau=0}=\hat{g}.$ Since $V$ is with the growth of $\frac{1}{\tau},$ we define a new variable $$u:=\tau V,$$ then $u|_{\tau=0}=1.$\\
We now impose the static Einstein vacuum condition $Ric(g)+ng=V^{-1}\nabla_{g}^{2}V$ and $\Delta_{g}V=nV$ on a triple $(M^n,g,V)$ of the above form.\\
One can decompose the tensor $Ric(g)+ng-V^{-1}\nabla_{g}^{2}V$ into components with respect to the product structure $N^{n-1}\times(0,\delta).$ A straightforward calculation shows that the vanishing of the component with both indices in $N^{n-1}$ is given by
\begin{equation}\label{ricdecompose}
\begin{split}
&R_{ij}(g_{\tau})-\frac{1}{2}g_{ij}^{''}+(\frac{n-2}{2\tau}-\frac{1}{4}g^{kl}g_{kl}^{'})g_{ij}^{'}+\frac{1}{2}g^{kl}g_{ik}^{'}g_{jl}^{'}+(\frac{1}{2\tau}g^{kl}g_{kl}^{'}+\frac{1}{\tau^{2}})g_{ij}\\
&=V^{-1}((\nabla_{g_{\tau}}^{2}V)_{ij}-\frac{V^{'}}{\tau}g_{ij}+\frac{V^{'}}{2}g_{ij}^{'}),
\end{split}
\end{equation}
where $g_{ij}$ denotes the tensor $g_{\tau}$ on $N^{n-1},$ $'$ denotes $\p_{\tau},$ and $R_{ij}(g_{\tau})$ denotes the $Ric$ tensor of $g_{\tau}$ with $\tau$ fixed.\\
Meanwhile, we also rewrite $\Delta_{g}V$ in terms of the metric $g=\tau^{-2}(d\tau^2+g_{\tau}),$ then we get
\begin{equation}\label{laplacianvminusnv}
\tau^{2}(V^{''}+\frac{1}{\tau}V^{'})+\tau^{2}g^{ij}(\nabla_{g}^{2}V)_{ij}=nV,
\end{equation}
where $g_{ij}$ still denotes the tensor $g_{\tau}$ on $N^{n-1}$ and $'$ denotes $\p_{\tau}.$\\
Substituting $u$ for $V,$ (\ref{ricdecompose}) turns to
\begin{equation}\label{ricug}
\begin{split}
&\tau ug_{ij}^{''}+(1-n)ug_{ij}^{'}-ug^{kl}g_{kl}^{'}g_{ij}-\tau ug^{kl}g_{ik}^{'}g_{jl}^{'}+\frac{\tau}{2}ug^{kl}g_{kl}^{'}g_{ij}^{'}\\
&-2\tau uR_{ij}(g_{\tau})+2\tau(\nabla_{g_{\tau}}^{2}u)_{ij}-2u^{'}g_{ij}+\tau u^{'}g_{ij}^{'}=0.\\
\end{split}
\end{equation}
And (\ref{laplacianvminusnv}) turns to
\begin{equation}\label{laplacianug}
\begin{split}
\tau\Delta_{g_{\tau}}u+\tau u^{''}-nu^{'}+\frac{\tau}{2}g^{ij}g_{ij}^{'}u^{'}-\frac{1}{2}g^{ij}g_{ij}^{'}u=0.
\end{split}
\end{equation}
Differentiating (\ref{ricug}) $m-1$ times with respect to $\tau$ and setting $\tau=0$ gives
\begin{equation}\label{expansiongderivative}
(m-n)\partial_{\tau}^{m}g_{ij}-g^{kl}(\partial_{\tau}^{m}g_{kl})g_{ij}-2(\partial_{\tau}^{m}u)g_{ij}=\nonumber
\end{equation}
(terms involving $\partial_{\tau}^{\mu}g_{ij}$ with $\mu<m$ and $\partial_{\tau}^{\nu}u$ with $\nu< m$).\\
Differentiating (\ref{laplacianug}) $m-1$ times with respect to $\tau$ and setting $\tau=0$ gives
\begin{equation}\label{expansionvderivative}
(m-1-n)\partial_{\tau}^{m}u-\frac{1}{2}g^{ij}(\partial_{\tau}^{m}g_{ij})=\nonumber
\end{equation}
(terms involving $\partial_{\tau}^{p}g_{ij}$ with $p<m$ and $\partial_{\tau}^{q}u$ with $q<m$).\\
So long as $m<n,$ we can inductively uniquely determine $\partial_{\tau}^{m}g_{ij}|_{\tau=0}$ and $\partial_{\tau}^{m}u|_{\tau=0}$ at each step if only we have $\hat{g}=g_{\tau}|_{\tau=0}.$ We compute that $$u^{'}|_{\tau=0}=0,~~~g_{ij}^{'}|_{\tau=0}=0,$$ $$u^{''}|_{\tau=0}=\frac{S(\hat{g})}{2(n-1)(n-2)},$$$$g_{ij}^{''}|_{\tau=0}=\frac{1}{2-n}[\frac{S(\hat{g})}{1-n}\hat{g}_{ij}+2Ric_{ij}(\hat{g})].$$
Note that $\partial_{\tau}^{m}V|_{\tau=0}$ is determined by $\partial_{\tau}^{(m+1)}u|_{\tau=0}.$ So we can uniquely determine $\partial_{\tau}^{m}g_{\tau}|_{\tau=0}$ and $\partial_{\tau}^{(m-1)}V|_{\tau=0}$ for any
$m<n$ if only we have $\hat{g}$.\\
When $m=n,$ we can only uniquely determine the term $[2\partial_{\tau}^{n}u-tr_{g_{\tau}}(\partial_{\tau}^{n}g_{\tau})]|_{\tau=0}.$
\end{proof}
\begin{lemm}\label{definingfct}
Suppose that $(M^n,g,V),$ $\hat{g}$ and $\tau$ are as that of Theorem \ref{expansion}. Let $(X^{n+1},h)=(\mathbb{S}^{1}\times M^n, h=V^{2}d\theta^{2}+g),$ then $\tau$ is the special defining function associated with the metric $d\theta^{2}+\hat{g}.$
\end{lemm}
\begin{proof}
By the assumption, we know that
\begin{equation}
g=\tau^{-2}(d\tau^{2}+g_{\tau})\nonumber
\end{equation}
in a neighborhood of conformal infinity $(N^{n-1},[\hat{g}]).$
We find that
\begin{equation}
\tau^{2}h=d\tau^{2}+[(\tau V)^{2}d\theta^{2}+g_{\tau}]\nonumber
\end{equation}
and
\begin{equation}
\hat{h}=\tau^{2}h|_{T(\mathbb{S}^{1}\times N^{n-1})}=[(\tau V)^{2}d\theta^{2}+g_{\tau}]|_{\tau=0}=d\theta^{2}+\hat{g}.\nonumber
\end{equation}
\end{proof}
Now that we have the above lemma, meanwhile we know that in fact the Riemannian manifold $(X^{n+1},h)$ is Einstein and satisfies $Ric(h)=-nh.$ By means of the expansion of an Einstein metric in \cite{G}, we arrive at the proof of the expansion of a static metric.
\begin{proof}[Proof of Theorem \ref{expansion}]
For a static Einstein vacuum $(M^n,g,V)$ satisfying the assumptions of Theorem \ref{expansion}, according to Lemma \ref{uniquelymnobigthann}, we may write
$$g=\hat{g}+g^{(1)}\tau+g^{(2)}\tau^{2}+\cdots+g^{(n-1)}\tau^{n-1}+o(\tau^{n-1}),$$
and
$$u=\tau V=1+u^{(1)}\tau+u^{(2)}\tau^{2}+\cdots+u^{(n-1)}\tau^{n-1}+o(\tau^{n-1}).$$
Since $h$ is an Einstein metric, by the results in \cite{G}, we know that $h_{\tau}$ satisfies the following:\\
When $n$ is odd,
\begin{equation}
h_{\tau}=\hat{h}+h^{(2)}\tau^{2}+(even~powers~of~\tau)+h^{(n-1)}\tau^{n-1}+h^{(n)}\tau^{n}+\cdots;\nonumber
\end{equation}
When $n$ is even,
\begin{equation}
h_{\tau}=\hat{h}+h^{(2)}\tau^{2}+(even~powers~of~\tau)+h^{(n)}\tau^{n}+s\tau^{n}\log\tau+\cdots;\nonumber
\end{equation}
where:\\
\begin{enumerate}
\item $h^{(2i)}$ are determined by $\hat{h}$ for $2i<n;$
\item $h^{(n)}$ is traceless when $n$ is odd;
\item the trace part of $h^{(n)}$ is determined by $\hat{h}$ and $s$ is traceless and determined by $\hat{h};$
\item the traceless part of $h^{(n)}$ is divergence free.
\end{enumerate}
As we have from Lemma \ref{definingfct} that $h_{\tau}=u^{2}d\theta^{2}+g_{\tau},$ we find that
$$h^{(k)}=\sum_{a+b=k,a\geq0,b\geq0}u^{(a)}u^{(b)}d\theta^{2}+g^{(k)}.$$
When $2i+1<n,$ $h^{(2i+1)}=0,$ which implies $$g^{(2i+1)}=\sum_{a+b=2i+1,a\geq0,b\geq0}u^{(a)}u^{(b)}=0.$$
Since we have $u^{(0)}=1,$ using
\begin{equation}
\begin{split}
&\sum_{a+b=2i+1,a\geq0,b\geq0}u^{(a)}u^{(b)}\\
&=\sum_{a+b=2i-1,a\geq0,b\geq0}u^{(a)}u^{(b)}+2u^{(0)}u^{(2i+1)}+2u^{(1)}u^{(2i)}\\
\end{split}\nonumber
\end{equation}
We conclude that $u^{(2i+1)}=0$ for $2i+1<n,$ which means $V^{(2i)}=0$ for $2i<n-1.$\\
Therefore we have
\begin{equation}
g_{\tau}=\hat{g}+g^{(2)}\tau^{2}+(even~powers~of~\tau)+g^{(2l)}\tau^{2l}+\cdots\nonumber
\end{equation}
\begin{equation}
V=\frac{1}{\tau}+V^{(1)}\tau+(odd~powers~of~\tau)+V^{(2l-1)}\tau^{2l-1}+\cdots\nonumber
\end{equation}
where $g^{(2l)}$ and $V^{(2l-1)}$ are uniquely determined by $\hat{g}$ for $2l\leq n-1.$

\end{proof}

\end{document}